\documentclass[a4paper,11pt]{article}
\usepackage[utf8]{inputenc}
\usepackage[british]{babel}
\usepackage[mono=false]{libertine}
\usepackage[T1]{fontenc}
\usepackage{amsthm}
\usepackage[top=2.5cm, bottom=2.5cm, left=2cm, right=2cm]{geometry}
\usepackage{xcolor}
\usepackage[font=small]{caption}
\usepackage{enumitem}

\newtheorem{theorem}{Theorem}
\newtheorem{lemma}{Lemma} 
\newtheorem{proposition}{Proposition}

\usepackage{titlesec}
\usepackage{slashbox}

\usepackage{cite}
\usepackage[pdftex]{graphicx}
\graphicspath{./figures}
\usepackage{array}
\usepackage{url}
\usepackage{mathtools}
\usepackage{amssymb,amsfonts,amsmath}
\usepackage{dsfont}
\usepackage{stmaryrd}
\usepackage{bm}
\usepackage[pdfpagemode=UseNone,bookmarksopen=false,colorlinks=true,urlcolor=blue,citecolor=magenta,citebordercolor=blue,linkcolor=blue]{hyperref}
\usepackage{footmisc}
\DefineFNsymbols{mySymbols}{{\ensuremath\dagger}{\ensuremath\Diamond}{\ensuremath{\ast}}{\ensuremath{\star}}}
\setfnsymbol{mySymbols}
\DeclareUnicodeCharacter{00A0}{~}

\def \({\left(}
\def \){\right)}
\def \[{\left[}
\def \]{\right]}

\newcommand{\be}{\begin{equation}}
\newcommand{\ee}{\end{equation}}

\newcommand{\bea}{\begin{align}}
\newcommand{\eea}{\end{align}}

\DeclareMathAlphabet{\varmathbb}{U}{bbold}{m}{n}

\usepackage{notations}
\usepackage{setspace}

\def\(({\left(}
\def\)){\right)}
\def\[[{\left[}
\def\]]{\right]}

\def\argmax{{\rm argmax}}

\newcommand{\BEAS}{\begin{eqnarray*}}
\newcommand{\EEAS}{\end{eqnarray*}}
\newcommand{\BEA}{\begin{eqnarray}}
\newcommand{\EEA}{\end{eqnarray}}

\def\(({\left(}
\def\)){\right)}                       
\def\[[{\left[}
\def\]]{\right]}

\usepackage{bm}
\graphicspath{{../figures/}}

\begin{document}
\title{Statistical and computational phase transitions \\ in spiked tensor estimation}

\author{Thibault Lesieur$^{\dagger}$,\, Léo Miolane$^{\Diamond}$,\, Marc Lelarge$^{\Diamond}$,\, Florent Krzakala$^{\star}$ \&\, Lenka Zdeborov\'a$^{\dagger}$}
\date{}
\maketitle
\let\thefootnote\relax\footnote{%
	\!\!\!\!\!\!\!\!\!\!\!\!\!This paper was presented at the IEEE International Symposium on Information Theory (ISIT) 2017 in Aachen, Germany.\\
	$\dagger$ Institut de Physique Th\'eorique, CNRS \& CEA \& Universit\'e Paris-Saclay, Saclay, France.\\
	$\star$ Laboratoire de Physique Statistique, CNRS \& Universit\'e Pierre et Marie Curie \& \'Ecole Normale Sup\'erieure \& PSL Universit\'e, Paris, France.\\
	$\Diamond$ D\'epartement d'Informatique de l'ENS, \'Ecole Normale Sup\'erieure \& CNRS \& PSL Research University \& Inria, Paris, France. \\
}

\begin{abstract}
	We consider tensor factorization using a generative
	model and a Bayesian approach. We compute rigorously the mutual
	information, the Minimal Mean Squared Error (MMSE), and unveil
	information-theoretic phase transitions. In addition, we study the
	performance of Approximate Message Passing (AMP) and show that it
	achieves the MMSE for a large set of parameters, and that
	factorization is algorithmically ``easy'' in a much wider region
	than previously believed. It exists, however, a ``hard'' region
	where AMP fails to reach the MMSE and we conjecture that no
	polynomial algorithm will improve on AMP.
\end{abstract}

\vspace{0.5cm}

This study inscribes into the line of research on low-rank tensor
decomposition, a problem with many applications ranging from signal
processing to machine
learning\cite{anandkumar2014tensor,cichocki2015tensor,sidiropoulos2016tensor}. We
consider the model of \cite{richard2014statistical} where the tensor
is a noisy version of a $r$-dimensional randomly generated spike and
analyze the Bayes-optimal inference of the spike, compute the
associated mutual information and the minimum mean-squared error
(MMSE). We also investigate whether the MMSE is achievable with some
known efficient algorithms, and most particularly by approximate
message passing (AMP).

\section{The spiked tensor model} 
One observes an order-$p$ tensor $Y \in \bigotimes^p \mathbb{R}^N$
created as
\begin{equation}
	Y = \frac{\sqrt{(p-1)!}}{N^{\frac{p-1}{2}}} \sum_{k=1}^r (X^{0}_k)^{\otimes
  p} + V \label{AdditiveNoise_Case} \, ,
\end{equation}
where $X^0_1, \dots, X^0_r \in \mathbb{R}^N$ are $r$ unknown vectors
to be inferred from $Y$, and $V \!\in \bigotimes^p \mathbb{R}^N$ is a
symmetric tensor accounting for noise. We denote by $X$ the
$N\!\times\! r$ matrix that collects the $r$ vectors $X_k$. The observed
tensor $Y$ can thus be seen as a rank $r$ perturbation of a random
symmetric tensor $V$. Consider now the setting where the $X^0$ is
generated at random from a known prior distribution. The core
question considered in this paper is: What is the best possible
reconstruction of $X^0$ one can hope for?

In fact, we can
look at even more general noise than just additive one as
in (\ref{AdditiveNoise_Case}). Denote for $i=1, \dots, N$,
$x_i=(x_{i,1}, \dots, x_{i,r}) \in \mathbb{R}^{r}$ the $r$-dimensional vector
created by aggregating the $i^{\rm th}$ coordinates of the $r$ vectors
$X_k$.  Assume that for $1 \leq i \leq N$ the $x_i^0$ are
generated independently from a probability distribution $P_X$ over $\R^r$. We denote, for $(i_1,i_2, \dots, i_p) \in \{1, \dots, N\}^p$
\begin{eqnarray}
W^0_{i_1, i_2, \dots, i_p} = \frac{\sqrt{(p-1)!}}{N^{\frac{p-1}{2}}}
\sum_{k=1}^r x^0_{i_1, k} x^0_{i_2, k} \cdots x^0_{i_p, k}  \, .
\label{Define_w}
\end{eqnarray}
For simplicity, we will assume to only observe the extra-diagonal elements of $Y$, i.e.\ the coefficients $Y_{i_1, i_2, \dots, i_p}$ for $1 \leq i_1 < \dots < i_p \leq N$. The case where all coefficients are observed can be directly deduced from this case.
The observed tensor $Y$ is generated from $W^0$ using a noisy
component-wise output channel $P_{\rm out}$ so that 
\begin{equation}
   P(Y|X^0) = \prod\limits_{i_1 < i_2 < \cdots < i_p} P_{\rm
   out}\left(Y_{i_1, i_2, \dots, i_p} \middle| W^0_{i_1, i_2,
     \dots, i_p}   \right)    \label{pout}\, .
\end{equation}
The simplest situation corresponds to
eq.~\eqref{AdditiveNoise_Case} with additive white Gaussian noise (AWGN), i.e.\ $P_{\rm out}( \, \cdot \, | \, w) = {\cal N}(w,\Delta)$. 

Given the above generative model and assuming that both the prior
distribution $P_X$ and the output channel $P_{\rm out}$ are known we
can write the Bayes-optimal estimator of $X^0$ as marginalization of
the following posterior distribution
\begin{equation}
P(X|Y) = \frac{1}{{\cal Z}_N}\prod_{i=1}^N P_{X}(x_i)  \! \! \! \! \! \! 
\prod\limits_{i_1 < i_2 < \cdots < i_p}  \! \! \! \! \! \!  P_{\rm
  out}\left( Y_{i_1, i_2, \dots, i_p} \middle| W_{i_1, i_2, \dots,
    i_p}\right)\, ,   \label{posterior}
\end{equation}
where ${\cal Z}_N$ is a normalization constant depending of the observed tensor
$Y$, $W_{i_1, i_2, \dots i_p}$ is defined analogously to \eqref{Define_w} (with $X$ instead of $X^0$).

We will study this tensor estimation problem in the limit where the
dimension $N \!\rightarrow \!\infty$ while the rank $r$ remains constant.
The factor $N^{\frac{p-1}{2}}$ is here to ensure that
information-theoretically the inference problem is neither trivially
hard nor trivially easy when one deals with signals such that $\|x_i\|$ and the noise magnitude are of order 1. The factor
$\sqrt{(p-1)!}$ is used for convenient rescaling of the signal-to-noise ratio.

\section{Related work and summary of results} 
Recently there have been numerous works on the matrix ($p=2$) version
of the above setting. In particular an explicit single-letter
characterization of the mutual information and of the optimal Bayesian
reconstruction error have been rigorously established
\cite{korada2009exact,DBLP:journals/corr/DeshpandeM14,krzakala2016mutual,barbier2016mutual,lelarge2016fundamental}.
A large part of these results rely on the approximate message passing
(AMP) algorithm. For the rank-one matrix estimation problems AMP has been
introduced by \cite{rangan2012iterative}, who also derived the state
evolution (SE) formula to analyze its performance, generalizing
techniques developed by \cite{bayati2011dynamics}. In
\cite{lesieur2015phase,lesieur2015mmse,Lesieurzdeborova2017} the
generalization to larger rank, and general output channel, was considered. Following the
theorem proven in \cite{DBLP:journals/corr/DeshpandeM14,barbier2016mutual,lelarge2016fundamental}, we
know that indeed AMP is Bayes-optimal and achieves the minimum
mean-squared error (MMSE) for a large set of parameters of the
problem. There, however, might exist a region denoted as {\it hard}, where this
is not the case, and polynomial
algorithms improving on AMP are not known.

In comparison, there has been much less work on Bayesian low-rank
tensor estimation. In statistical physics, the measure
(\ref{posterior}) was considered for $Y$ with random i.i.d.\
components. For a Gaussian $P_X$, it is called the spherical $p$-spin
glass \cite{crisanti1992sphericalp}, while for Rademacher $P_X$ it is
the Ising $p$-spin glass \cite{MezardParisi87b}.  AMP for tensor
estimation is actually equivalent to the so-called
Thouless-Anderson-Palmer equations in spin glass theory
\cite{ThoulessAnderson77,crisanti1995thouless,korada2009exact}. In the
context of tensor PCA these equations have been studied by Richard and
Montanari \cite{richard2014statistical} for the maximum likelihood
estimation. Interestingly, they showed that the {\it hard} phase
was particularly large in the tensor estimation case and that, with
side information (such that for each component $x_{i,k}$ we have its
{\it direct} noisy observation), the estimation problem becomes
easier. However, such a kind of component-wise side information is
very strong and rarely available in applications. 
The tight statistical limits for the present tensor model were also studied in
\cite{korada2009exact} for the special case of the Rademacher (Ising) prior. For more
generic priors only upper and lower bounds are known rigorously  
\cite{perry2016statistical}.
\\



{\bf Summary of results:} 
In this contribution, we aim to bridge the
gap between what is known for the general $r, P_x, P_{\rm out}$
Bayesian estimation for low-rank matrices and what is known for
low-rank tensors.  We present the following contributions: 
\begin{enumerate}[label=(\Alph*),noitemsep]
	\item The AMP algorithm and its state evolution analysis for the Bayes-optimal
	tensor estimation, see sections~\ref{sec:amp}~and~\ref{sec:theory}.
	\item The so-called {\it channel universality}
		result that allows us to map any generic channel $P_{\rm out}$ on a
		model with additive Gaussian noise, see section~\ref{sec:amp}.
	\item Rigorous formula for the
		asymptotic mutual information and the MMSE, thus
		generalizing the matrix results of
		\cite{barbier2016mutual,lelarge2016fundamental}, see section~\ref{proof}.  
	\item The identification of statistical and computational phase transitions. In
		fact, we show that as soon as the effect of a non-zero-mean prior is
		taken properly into account, the hard region shrinks considerably,
		making the tensor decomposition problem much easier than hitherto
		believed, at least for algorithms that do take the prior information into
		account.  Having a reliable prior information on the distribution of
		$x_{i}$ (not on each of the components as in
		\cite{richard2014statistical}) is rather realistic in applications,
		for instance when constraints of negativity or membership to
		clusters are imposed. This is presented in sections~\ref{sec:theory}~and~\ref{sec_examples}.
\end{enumerate}


\section{AMP algorithm \& channel universality}
\label{sec:amp}

We discuss in this section the Approximate Message Passing (AMP)
algorithm for the Bayesian version of the problem. This is a
relatively straightforward generalization of what has been done for
the low-rank matrix estimation in
e.g. \cite{rangan2012iterative,NIPS2013_5074,Lesieurzdeborova2017}, i.e.\ $p=2$ case of the
present setting. In general, AMP is derived from belief propagation by
taking into account that every variable in the corresponding
graphical model has a large number of neighbors. Since the incoming
messages are considered independent one can use the central limit
theorem and represent each message as a Gaussian with a given mean
$\hat x_i \in \mathbb{R}^{r}$ and covariance
$\sigma_i \in \mathbb{R}^{r\times r}$.

A crucial property, called {\it channel universality}, that the
tensor-AMP shares with the low-rank matrix estimation, allows to
drastically simplify the problem of tensor estimation with generic
output channel $P_{\rm out}$. The justification of this property 
follows closely the low-rank matrix estimation case, and we refer
the reader to
\cite{lesieur2015mmse,krzakala2016mutual,Lesieurzdeborova2017}. First,
we define
the Fisher score tensor $S$ associated to the output channel $P_{\rm out}$
and its Fisher information $\Delta$ as
\begin{eqnarray}
       S &\equiv& \left.\frac{\partial \log{P_{\rm out}(Y,w)}}{ \partial w}
	   \right|_{w=0} \, , \label{def_S}\\
      \frac{1}{\Delta} &\equiv& {\mathbb E}_{Y \sim P_{\rm out}(\cdot \, | \, 0)} \left[
		  \left( \frac{\partial \log{P_{\rm out}(Y,w)}}{ \partial w} \right) _{\!\!w=0}^{\!2}
  \right] \, .
\end{eqnarray}
where it is understood in \eqref{def_S} that the function $y \mapsto \frac{\partial \log{P_{\rm out}(y,w)}}{ \partial w} \big|_{w=0}$ acts component-wise on $Y$.
Informally speaking, the channel universality property states that the
mutual information of the problem defined by the output channel
$P_{\rm out}$ is the same as the one of a AWGN \eqref{AdditiveNoise_Case} with variance $\Delta$,
and that the AMP algorithm written for the Bayes-optimal inference of
low-rank tensors then depends on the data tensor $Y$ and the output
channel $P_{\rm out}$ only trough the tensor $S$ and the effective
noise $\Delta$.
\\


AMP involves an auxiliary function that depends
explicitly on the prior as follows. Define the probability
distribution
\begin{equation}
   {\cal M}(x) =  \frac{1}{{\cal Z}_X(A,B)} \,P_X(x) e^{B^\top x -
     \frac{x^\top A x}{2}}\, , \label{F_in_x}
\end{equation}
where ${\cal Z}_X(A,B)$ is a normalization factor. Then AMP uses the
function
${f_{\rm in}}(A,B) \in \mathbb{R}^{r}, A \in \mathbb{R}^{r \times r},
B \in \mathbb{R}^{r}$
defined by the expectation
$f_{\rm in}(A,B) ={\mathbb{E}}_ {{\cal M}(x)} [x]$ as well as the
covariance matrix $\partial_B {f_{\rm in}}(A,B)$.  We shall denote
the \textit{overlap} of $u =(u_1, \dots u_N), v=(v_1,\dots v_N) \in (\R^r)^N$ by
$$u \cdot v = \frac{1}{N} \sum_{j=1}^n u_j v_j^{\top} \in \R^{r \times r} \,.$$
AMP is then written as an
iterative update procedure on the estimates of the posterior means and
co-variances $\hat x_i$ and $\sigma_i$ that uses auxiliary variables
$B_i \in \mathbb{R}^{r}$ and $A \in \mathbb{R}^{r\times r}$:
\begin{align}
B_i^{t} &= \frac{\sqrt{(p-1)!}}{N^{\frac{p-1}{2}}} \!\!\!\!\! \sum\limits_{i_2 <
  i_3 < \cdots < i_p} \!\!\!\!  S_{i, i_2, i_3 \cdots i_p} \, \hat{x}_{i_2}^t \circ
  \hat{x}_{i_3}^t \circ \cdots \circ \hat{x}_{i_p}^t  
- \frac{(p-1)}{{\Delta}}\!
\bigg[  \frac{1}{N} \sum_{j=1}^N \sigma_{j}^t \circ 
 (\hat{x}^t \cdot \hat{x}^{t-1})^{\circ(p-2)} \bigg]\, 
  \hat x_i^{t-1} \\
A^t &= \frac{1}{{\Delta}} 
(\hat{x}^t \cdot \hat{x}^t)^{\circ(p-1)}
\\
\hat{x}_i^{t+1}	&= f_{\rm in}(A^t,B_i^t)\label{AMP_x}
\\
\sigma_i^{t+1}	&=\partial_B {f_{\rm in}}(A^t,B_i^t)\, ,
\end{align}
where $\circ$ denotes a component-wise (Hadamard) product of matrices,
and $x^{\circ p}$ the corresponding component-wise power.

\section{Theoretical analysis}
\label{sec:theory}
\subsection{State evolution of AMP}
The evolution of the AMP algorithm in the limit of large systems
$N\to \infty$ can be tracked via a low-dimensional set of
equations called the {\it state evolution} (SE). For maximum-likelihood estimation the
state evolution have been used in \cite{richard2014statistical}. Its
heuristic derivation for the present case of general rank $r$, prior
$P_X$, and output $P_{\rm out}$ follows line by line the matrix
estimation case detailed in \cite{Lesieurzdeborova2017}.

For the Bayes-optimal inference, SE is written in terms of an order parameter
$M^t \in {\mathbb R}^{r\times r}$  describing the overlap between $\hat{x}^t$ (the AMP
estimator at iteration $t$) and the ground truth $x^0$
defined as
$M^t = \hat{x}^t \cdot x^{0}$, and reads
\begin{eqnarray}
\label{SE_M1}
&M^{t+1}= \mathbb{E}_{Z,x_0}\left[ f_{\rm in}\left( \widehat M^t, \widehat M^t x_0 + \big(\widehat M^t\big)^{1/2} Z \right)x_0^\top \right] 
, \\ 
 & \widehat M^t=  {{(M^t)}^{\circ(p-1)}}/{{\Delta}} \, ,\label{SE_M2}
\end{eqnarray}
where $Z \sim \cN(0,I_r)$ and $x_0 \sim P_X$ are independent. $M^{\circ(n)}$ is again the $n$-th Hadamard power of a matrix $M$. 

We shall not present a rigorous proof of the SE for tensor estimation
and rely instead on standard arguments from statistical physics. The
performance of the AMP algorithm can be understood by initializing the
SE at $M^{t=0} = 0$. Or when $M=0$ is a fixed point of SE we
initialize as $M^{t=0} = \epsilon$, an infinitesimally small number
(accounting for the fact that a random initialization of AMP will
---due to finite size fluctuations--- be infinitesimally correlated
with the ground truth). We denote $M_{\rm AMP}$ the fixed point of the
state evolution resulting from iterations of (\ref{SE_M1}-\ref{SE_M2})
from this initialization. The mean-squared error achieved by 
tensor-AMP is then
\begin{equation}
   {\rm MSE}_{\rm AMP} =    {\rm Tr} \left[  \Sigma_X - M_{\rm AMP} \right] \, .\label{MSE}
\end{equation}
where $\Sigma_X = \mathbb{E}_{x}[x x^{\top}]$. When $P_X$ has zero mean, this is the covariance matrix of $P_X$.

\subsection{Information-theoretically optimal inference}
Our next goal is to analyze the performance of (possibly intractable)
Bayes-optimal inference that evaluates the marginals of the posterior
probability distribution (\ref{posterior}). The error achieved by this
procedure will be denoted the minimum mean-squared error
(MMSE) and is formally defined as
$$
\text{MMSE}_N = \inf_{\hat{\theta}} 
\left\{
\frac{1}{N}\E \left[
		\left\|X^0 - \hat{\theta}(Y) \right\|^2
\right]\right\} = 
\frac{1}{N}\E \left[
	\left\|X^0 - \E[X^0|Y] \right\|^2
\right]\,,
$$
where the infimum is taken over all measurable functions $\hat{\theta}$ of $Y$.
In order to compute the MMSE it is instrumental to compute
the mutual information $I(X^0;Y)$.
This quantity is related to the free energy from statistical physics (see section \ref{proof} and \cite{krzakala2016mutual}).
To compute the limit of such quantities, 
one traditionally applies the replica method stemming from
statistical physics \cite{MezardParisi87b}. We take advantage of the
fact that for the Bayes-optimal inference the so-called {\it replica
symmetric} version of this method yields the correct free energy
\cite{zdeborova2015statistical}. The replica method yields 
\begin{eqnarray} \label{limit_I}
	\frac{1}{N} I(X^0;Y) \xrightarrow[N \to
        \infty]{}\frac{1}{2p\Delta}   \sum_{k,k'=1}^r    (\Sigma_X)_{k,k'}^p
		- \sup_{M \in S_r^+} \phi_{\text{RS}}(M)
  \, ,  \\
\phi_{\rm RS}(M) = \displaystyle \mathop{\mathbb{E}}_{Z,x_0} \left[ \log{ {\cal Z}_X
\left( \widehat M  , \widehat M  x_0 + \left(\widehat M \right)^{\! 1/2}\! Z \right)}   \right]
                      - \frac{p-1}{2p \Delta}
  \sum_{k,k'=1}^r   {M}_{k k'}^p
 \label{BetheFreeEnergy}
\end{eqnarray}
where $\widehat M = {M}^{\circ(p-1)}/{\Delta}$, $ {\cal Z}_X(A,B)$
is defined in eq.~\eqref{F_in_x}, $x_0 \sim P_X$ and $Z \sim  \cN(0,I_r)$ are independent random variables. $S_r^+$ denotes the set of $r \times r$ symmetric positive semi-definite matrices. In section~\ref{proof} we prove this result for the rank-one case ($r=1$).

The replica free energy (\ref{BetheFreeEnergy}) not only provides the
limit of the mutual information $I(X^0;Y)$,
but thanks to an ``I-MMSE Theorem'' (similar to \cite{guo2005mutual})
it yields the value of the MMSE for tensor estimation, see sec.~\ref{proof}. Denoting
$M^* = \argmax_{M} \phi_{\rm RS}(M)$ we get 
\begin{equation}\label{MMSE}
	{\rm MMSE} = \lim_{N \to \infty} \text{MMSE}_N 
	=    {\rm Tr} \left[   \Sigma_X - M^* \right] 
	\, . 
\end{equation}
We proved \eqref{MMSE} rigorously, but only in the rank-one case and for odd values of $p$, see
again sec.~\ref{proof}. 
Notice that when $r \geq 2$ the estimation problem is symmetric under permutations of the $r$ columns of $X^0$: \eqref{MMSE} is not expected to be true without further assumptions.

\subsection{Statistical and computational trade-off} 

By evaluation of the derivative of (\ref{BetheFreeEnergy}) with
respect to $M$ one can check that critical points of
(\ref{BetheFreeEnergy}) are fixed points of the state evolution
equations (\ref{SE_M1}-\ref{SE_M2}) allowing all the results to be
read of the curve $\phi_{\rm RS}(M)$: The global maximum of
(\ref{BetheFreeEnergy}) gives the MMSE while the
(possibly local) maximum reached by iteration of
(\ref{SE_M1}-\ref{SE_M2}) from the uninformative initialization
yields the ${\rm MSE}_{\rm AMP}$.

We now discuss the interplay between the MMSE and ${\rm MSE}_{\rm
  AMP}$. The working hypothesis in this paper is
that AMP yields lowest MSE among known polynomial 
algorithms. Depending on the parameters of model (\ref{posterior}), 
i.e.\ the order of the tensor $p$, rank $r$, prior distribution
$P_X$, and output channel $P_{\rm out}$ that appears in the SE only via its
Fisher information $\Delta$, we can distinguish between two cases: 
the {\bf easy} phase where asymptotically AMP is Bayes optimal so that
${\rm MMSE} = {\rm MSE}_{\rm AMP}$, and the {\bf hard} phase where
${\rm MMSE} < {\rm MSE}_{\rm AMP}$.

Given both the ${\rm MMSE}$ and ${\rm MSE}_{\rm AMP}$ are
non-decreasing in $\Delta$ we denote the borders of the hard phase
(when it exists) as follows: {\bf Information theoretic threshold}
$\Delta_{\rm IT}$ as the (limsup of the) highest $\Delta$ for which
${\rm MMSE} < {\rm MSE}_{\rm AMP}$.  {\bf Algorithmic threshold}
$\Delta_{\rm Alg}$ as the (liminf of the) lowest $\Delta$ for which
${\rm MMSE} < {\rm MSE}_{\rm AMP}$.  Another threshold used in this
paper is that of a critical value $\Delta_c$ defined as smallest
$\Delta$ such that for $\Delta > \Delta_c$ one has $M_{\rm AMP}= M^*(\Delta = +\infty)$ (the estimate one can do when the noise is infinite),
and for $\Delta< \Delta_c$ one has $M_{\rm AMP}>M^*(\Delta = +\infty)$. Note that from the
definition we must have $\Delta_c \ge \Delta_{\rm Alg}$. In cases
where the hard phase does not exist, but $\Delta_c< \infty$ we will
consider that $\Delta_c=\Delta_{\rm IT}=\Delta_{\rm Alg}$.

Existing results on maximum likelihood estimation \cite{richard2014statistical} suggest
that for tensor decomposition $p\ge 3$ we have
$\Delta_{\rm Alg}=\Delta_c=0$ in the limit $N\to 0$
considered in this paper.
This means that the spiked model
of low-rank tensor decomposition is algorithmically very hard,
compared to matrix $p=2$ case.  The authors of
\cite{richard2014statistical}  give a good account on how $\Delta$
needs to scale with $N$ for known polynomial algorithms to work.  

For the Bayes-optimal estimation the situation seems at first sight
similar. Indeed, whenever the prior $P_X$ has a zero mean, for $p\ge 3$ we get
$\Delta_{\rm Alg} = \Delta_c=0$ and the hard phase is consequently huge. This
can be seen as follows. Indeed if the mean of the prior $P_X$ is zero
then the state evolution equations (\ref{SE_M1}-\ref{SE_M2}) have a
fixed point $M=0$.  Expanding the state evolution around this fixed
point we find
\begin{equation}
M^{t+1} =  \frac{1}{\Delta} \Sigma_X
  \left[{(M^t)}^{\circ(p-1)} \right] \Sigma_X \, .
\end{equation}
Whenever $p\ge 3$ the fixed point $M=0$ is stable for all $\Delta>0$. Hence $\Delta_{\rm Alg}=\Delta_c=0$ for priors of
zero mean. 

A closer look, however, shows that the situation is not so pessimistic.  Indeed,
as soon as the mean of the prior $P_X$ is non-zero, $M=0$ is no longer
a fixed point of the state evolution and once we solve the state
evolution equations we observe either $\Delta_{\rm Alg}>0$ (with AMP
performing optimally for $\Delta<\Delta_{\rm Alg}$) or the hard phase
is completely absent and AMP has information-theoretically optimal
performance for all $\Delta$.  We give examples of such priors in
section \ref{sec_examples}.

\section{Rigorous results}
\label{proof}
We present in this section rigorous results for the rank-one case ($r=1$).
As mentioned above, the universality property \cite{lesieur2015mmse,krzakala2016mutual} reduces the computation of the mutual information to the case of additive white Gaussian noise.

Consider a probability distribution $P_X$ over $\R$ that admits a finite second moment $\Sigma_X$.
The observation model~\eqref{AdditiveNoise_Case} reduces in the rank-one case to
$$
			Y_{i_1,\dots, i_p} = \frac{\sqrt{(p-1)!}}{N^{(p-1)/2}} x^0_{i_1} \dots x^0_{i_p} + V_{i_1, \dots, i_p} 
			\quad \text{for} \ 1 \leq i_1 < \dots < i_p \leq N,
$$
where $X^0=(x_1^0, \dots, x_N^0) \iid P_X$ and $(V_{i_1,\dots,i_p})_{i_1 < \dots < i_p} \iid \cN(0,\Delta)$ are independent.
We define the Hamiltonian
\begin{align}
	H_{N}(X)
	=\Delta^{-1}  \sum_{i_1< \dots < i_p}
	\frac{\sqrt{(p-1)!}}{N^{(p-1)/2}} \, Y_{i_1, \dots, i_p} \,
	x_{i_1}\dots x_{i_p} 
	-\frac{(p-1)!}{2N^{p-1}} 
	\left(x_{i_1} \dots x_{i_p}\right)^2 \,,
	\label{Hamiltonian}
\end{align}
for $X = (x_1, \dots, x_N) \in \R^N$. We also write $dP_X(X)=\prod_{i=1}^N dP_X(x_i)$.
The posterior distribution \eqref{posterior} of $X^0$ given $Y$ reads then:
\begin{equation}
	dP(X^0 = X|Y) = \frac{1}{\mathcal{Z}_N} dP_X(X) e^{H_N(X)} \,,
	\label{statphys}
\end{equation}
where ${\mathcal Z}_N$ is the appropriate normalizing factor.
Then the free energy is defined as (minus) the logarithm of ${\mathcal Z}_N$ of the Boltzmann probability 
divided by $N$ and averaged over $Y$. This is of particular interest
since it is related to the  mutual~information (see \cite{krzakala2016mutual}):
$$
I(X^0;Y) = \frac{N}{2p\Delta}
\Sigma_X^{p}
- \mathbb{E} [\log \mathcal{Z}_N] + O(1) \,.
$$
In the rank-one case, the expression \eqref{BetheFreeEnergy} of $\phi_{\rm RS}$ simplifies, so that we will use in this section
\begin{equation} \label{phi_rs_1}
\phi_{\rm RS} : m \geq 0 \mapsto \E \left[ \log \int dP_X(x) \exp\left(
		\sqrt{\frac{m^{p-1}}{\Delta}} Z x + \frac{m^{p-1}}{\Delta} x x^0 - \frac{m^{p-1}}{2 \Delta} x^2
\right) \right]
- \frac{p-1}{2p \Delta} m^p \,,
\end{equation}
where $\E$ is the expectation with respect to the independent random variables $x_0 \sim P_X$ and $Z \sim \cN(0,1)$.
The proof of \eqref{limit_I} reduces then to the following Theorem.
\begin{theorem}[Replica-symmetric formula for the free energy]\label{thm1}
	Let $P_X$ be a probability distribution over \,$\R$, with finite second moment. Then, for all $\Delta>0$
	\begin{equation} \label{eq:lim_F}
		F_N \equiv \frac{1}{N} \mathbb{E} \left[ \log {\mathcal Z}_N \right] \xrightarrow[N \to \infty]{} \sup_{m \geq 0} \phi_{\rm RS}(m) \equiv F_{\rm RS}(\Delta) \,.
	\end{equation}
\end{theorem}
We now define the tensor-MMSE, $\text{T-MMSE}_N$ by
\begin{align*}
	\text{T-MMSE}_N(\Delta)=
	\inf_{\hat{\theta}} \left\{ 
		\frac{p!}{N^p} \sum_{i_1 < \dots < i_p} \left(
			x^0_{i_1} \dots x^0_{i_p}
			- \hat{\theta}(Y)_{i_1 \dots i_p} 
		\right)^2
	\right\} \,,
\end{align*}
where the infimum is taken over all measurable functions $\hat{\theta}$ of the observations $Y$.

Let us write $\lambda = \frac{1}{\Delta}$. Using an ``I-MMSE Theorem'' (see \cite{guo2005mutual}) and the fact that the tensor MMSE is achieved by the posterior mean of $(X^0)^{\otimes p}$ given $Y$, it is not difficult to verify that
$$
\frac{\partial F_N}{\partial \lambda} =
\frac{N(N-1) \dots (N-p+1)}{2 p N^p} \big(\Sigma_X^p - \text{T-MMSE}_N(\Delta)\big).
$$
The arguments are the same than in the matrix ($p=2$) case, see \cite{lelarge2016fundamental} Corollary 17.
$\text{T-MMSE}(\Delta)$ increases with the noise level $\Delta$, so that $\frac{\partial}{\partial \lambda} F_N$ is a non-decreasing function of $\lambda$. $F_N$ is thus a convex function of $\lambda$, and so is $F_{\rm RS}$ its pointwise limit. 
Consequently,
$\frac{\partial}{\partial \lambda} F_N \to
\frac{\partial}{\partial \lambda} F_{\rm RS}$ at all values of $\lambda$ at which $F_{\rm RS}$ is differentiable, that is for almost every $\Delta > 0$. For these values of $\Delta$, one can also verify that the maximizer $m^*$ of $\phi_{\rm RS}$ is unique: we refer to \cite{lelarge2016fundamental} for a detailed proof in the matrix case $p=2$.
We thus obtain the following theorem:
\begin{theorem}[{\rm Tensor-MMSE}]\label{thm2}
	For almost every $\Delta >0$, $\phi_{\rm RS}$ admits a unique maximizer $m^*(\Delta)$ over $\R_+$ and
	$$
	\text{T-MMSE}_N \xrightarrow[N \to \infty]{} 
	\Sigma_X^p - m^*(\Delta)^p \,.
	$$
\end{theorem}

The information-theoretic threshold $\Delta_{\rm IT}$ is the maximal value of $\Delta$ such that ${\lim \text{T-MMSE}_N < \Sigma_{X}^p - \E_{P_X}[x]^{2p}}$ (which is the asymptotic performance achieved by random guess). We obtain thus the precise location of the information-theoretic threshold:
$$
\Delta_{\rm IT} = \sup \left\{ \Delta > 0 \, \middle| \, m^*(\Delta) > \E_{P_X}[x]^2 \right\}\,.
$$
Let $X=(x_1, \dots, x_N)$ be a sample from the posterior \eqref{posterior}, independently of everything else.
An extension of Theorem~2 of \cite{lelarge2016fundamental} (that was derived for priors $P_X$ with bounded support) to the tensor case, gives that for almost every $\Delta > 0$,
\begin{equation} \label{eq:overlap_p}
	\mathbb{E} \left|
	\left( \frac{1}{N}\sum\limits_{i=1}^N x_{i}^0 x_{i} \right)^{\!\! p}- m^*(\Delta)^p
	\right|
	\xrightarrow[N \to \infty]{} 0 \,,
\end{equation}
i.e.\ the $p^{\text{th}}$-power of the overlap $X \cdot X^0$ concentrates around $m^*$. This leads to
\begin{theorem}[Vector-MMSE for odd $p$]\label{thm3}
	Suppose that $P_X$ has a bounded support.
	If $p$ is odd, then for almost every $\Delta >0$
	$$
	\text{MMSE}_N \xrightarrow[N \to \infty]{} \Sigma_X  - m^*(\Delta).
	$$
\end{theorem}

Before showing how \eqref{eq:overlap_p} implies Theorem~\ref{thm3} we need to introduce a fundamental property of Bayesian inference: the Nishimori identity.
\begin{proposition}[Nishimori identity]\label{th:nishimori}
	Let $(X,Y)$ be a couple of random variables on a polish space. Let $k \geq 1$ and let $X^{(1)}, \dots, X^{(k)}$ be $k$ i.i.d.\ samples (given $Y$) from the distribution $P(X=\cdot \, | \, Y)$, independently of every other random variables. Let us denote $\langle \cdot \rangle$ the expectation with respect to $P(X=\cdot \, | \, Y)$ and $\mathbb{E}$ the expectation with respect to $(X,Y)$. Then, for all continuous bounded function $f$
	$$
	\mathbb{E} \langle f(Y,X^{(1)}, \dots, X^{(k)}) \rangle
	=
	\mathbb{E} \langle f(Y,X^{(1)}, \dots, X^{(k-1)}, X) \rangle \,.
	$$
\end{proposition}

\begin{proof}
	It is equivalent to sample the couple $(X,Y)$ according to its joint distribution or to sample first $Y$ according to its marginal distribution and then to sample $X$ conditionally to $Y$ from its conditional distribution $P(X=\cdot \, | \, Y)$. Thus the $(k+1)$-tuple $(Y,X^{(1)}, \dots,X^{(k)})$ is equal in law to $(Y,X^{(1)},\dots,X^{(k-1)},X)$.
\end{proof}

We will now use Proposition~\ref{th:nishimori} to prove Theorem~\ref{thm3}.

\begin{proof}[Proof of Theorem~\ref{thm3}]
	Let $\langle \cdot \rangle$ denote the expectation with respect to the posterior distribution $P(X^0=\cdot \ | \ Y)$, and let $X$ be a sample from this distribution, independently of everything else.
	The best estimator of $X^0$ in term of mean-squared error is the posterior mean $\langle X \rangle = (\langle x_1 \rangle, \dots, \langle x_N \rangle)$. Therefore
	\begin{align*}
		\text{MMSE}_N 
		&= \frac{1}{N} \mathbb{E}\left[
		\sum_{i=1}^N (x^0_i - \langle x_i \rangle)^2
	\right]
	= \frac{1}{N} \mathbb{E}\left[
		\sum_{i=1}^N (x^0_i)^2 + \langle x_i \rangle^2 - 2 \langle x_i^0 x_i \rangle
	\right]
	\\
	&= \Sigma_X + \mathbb{E} \langle X \cdot X' \rangle
	- 2 \mathbb{E} \langle X^0 \cdot X \rangle \,,
\end{align*}
where $X'$ is another sample from $\langle \cdot \rangle$, independently of everything else.
We apply now the Nishimori identity (Proposition~\ref{th:nishimori}) to obtain $\mathbb{E} \langle X \cdot X' \rangle = \mathbb{E} \langle X^0 \cdot X \rangle$. This gives
$$
\text{MMSE}_N = \Sigma_X - \mathbb{E} \langle X \cdot X^0 \rangle \,.
$$
We then deduce from \eqref{eq:overlap_p} that $\mathbb{E} \langle X \cdot X^0 \rangle \xrightarrow[N \to \infty]{} m^*$, because $p$ is here supposed to be odd. This concludes the proof.
\end{proof}

We will now prove Theorem~\ref{thm1}.  For the matrix case ($p=2$), this has been proved in \cite{krzakala2016mutual,
barbier2016mutual,lelarge2016fundamental} and we explain here how this can be adapted to the case $p \geq 2$. 
To prove the limit \eqref{eq:lim_F}, one shows successively an upper
bound on $\limsup F_N$ and the matching lower bound on $\liminf
F_N$. 
As shown in \cite{lelarge2016fundamental} (Section 6.2.2) one only need to prove Theorem~\ref{thm1} for input distributions $P_X$ with finite support $S$. We now assume to be in this situation.

\subsection{Adding a small perturbation} \label{sec:perturbation}

One of the key ingredient of the proof is the introduction of a small perturbation of our model, that takes the form of a small amount of side information.
This kind of techniques are frequently used for the study of spin glasses, where these small perturbations forces the Gibbs measure to verify some crucial identities, see \cite{panchenko2013sherrington}. In our context of Bayesian inference, we will see that small quantities of side information ``breaks'' the correlations of the signal variables under the posterior distribution.

Let us fix $\epsilon\in [0,1]$, and suppose we have access to the additional information, for $1 \leq i \leq N$
\begin{equation} \label{eq:pert}
	Y'_i =
	\begin{cases}
		x^0_i &\text{if } L_i = 1, \\
		* &\text{if } L_i = 0,
	\end{cases}
\end{equation}
where $L_i \iid \text{Ber}(\epsilon)$ and $*$ is a symbol that does not belong to $\R$. 
The posterior distribution of $X$ is now
$$
P(X| Y,Y') = \frac{1}{\cZ_{N,\epsilon}} \left(\prod_{i | Y'_i \neq *}1(x_i=Y'_i) \right)\left( \prod_{i | Y_i'=*} P_X(x_i) \right) e^{H_N(X)} \,,
$$
where $\cZ_{N,\epsilon}$ is the appropriate normalization constant.
For $X =(x_1, \dots, x_N)\in \R^N$ we will use the notation
\begin{equation} \label{eq:bar}
	\bar{X} = (L_1 x^0_1 + (1-L_1) x_1, \dots, L_N x^0_N + (1-L_N)x_N) \,.
\end{equation}
$\bar{X}$ is thus obtained by replacing the coordinates of $X$ that are revealed by $Y'$ by their revealed values. The notation $\bar{X}$ will allow us to obtain a convenient expression for the free energy of the perturbed model
$$
F_{N,\epsilon} = \frac{1}{N} \mathbb{E} \log \cZ_{N,\epsilon} = \frac{1}{N} \mathbb{E} \Big[ \log \sum_{X \in S^N} P_X(X) e^{H_{N}(\bar{X})}\Big] \,.
$$
The next lemma shows that the perturbation does not change the free energy up to the order of $\epsilon$. Recall that we supposed the support $S$ of $P_X$ to be finite, so we can find a constant $K$ such that $S \subset [-K,K]$.

\begin{lemma} \label{lem:approximation_f_n_epsilon}
	For all $n \geq 1$ and all $\epsilon,\epsilon' \in [0,1]$, we have
	$$
	|F_{N,\epsilon} - F_{N,\epsilon'} | \leq \frac{K^{2p}}{\Delta} |\epsilon - \epsilon'|.
	$$
\end{lemma}

Lemma~\ref{lem:approximation_f_n_epsilon} follows from a direct adaptation of Proposition~23 from \cite{lelarge2016fundamental} to the tensor case.
Consequently, if we suppose $\epsilon \sim \mathcal{U}([0,1])$ and define $\epsilon_N = N^{-1/2} \epsilon$ and $L_i \iid \text{Ber}(\epsilon_N)$, independently of everything else, we have
\begin{equation} \label{eq:pert_f}
	|F_N - \mathbb{E}_{\epsilon}[F_{N,\epsilon_N}] | \xrightarrow[N \to \infty]{} 0 \,,
\end{equation}
where $\mathbb{E}_{\epsilon}$ denotes the expectation with respect to $\epsilon$ only. It remains therefore to compute the limit of the free energy under a small perturbation.
As shown in \cite{andrea2008estimating}, the perturbation \eqref{eq:pert} forces the correlations to vanish asymptotically. 

\begin{lemma}[Lemma 3.1 from \cite{andrea2008estimating}]\label{lem:correlations}
	$$
	\mathbb{E}_{\epsilon}  \left[ \frac{1}{N^2} \sum_{1 \leq i,j \leq N} I(x^0_i;x^0_j \, | \, Y,Y') \right] \leq \frac{2 H(P_X)}{\sqrt{N}} \,.
	$$
\end{lemma}
Let us write $\langle \cdot \rangle$ the expectation with respect to $P(X= \cdot \ | \ Y,Y')$, and let $X^{(1)}, X^{(2)}$ be two independents samples from $P(X= \cdot \ | \ Y,Y')$, independently of everything else. We define $ Q = \langle X^{(1)} \cdot X^{(2)} \rangle$. Notice that $Q$ is a non-negative random variable.
As a consequence of Lemma~\ref{lem:correlations}, the overlaps under the posterior distribution concentrates around $Q$:
\begin{lemma} \label{lem:overlap_concentration}
	\begin{align}
		&\mathbb{E} \left\langle \left( X^{(1)} \cdot X^{(2)} - Q \right)^2 \right\rangle \xrightarrow[N \to \infty]{} 0 \quad \text{and} \quad
		\mathbb{E} \left\langle \left( X^{(1)} \cdot X^0 - Q \right)^2 \right\rangle \xrightarrow[N \to \infty]{} 0 \,,
		\label{eq:ov1}
	\end{align}
	where $\mathbb{E}$ denotes the expectation with respect all random variables.
\end{lemma}
Lemma~\ref{lem:overlap_concentration} follows from the arguments of section 4.4 from \cite{lelarge2016fundamental}.

The arguments presented in this section are robust and apply to a large class of Hamiltonians. In particular, we will be able to apply in the sequel Lemmas~\ref{lem:approximation_f_n_epsilon} and~\ref{lem:overlap_concentration} to other Hamiltonians and posterior distributions (and corresponding free energies).

\subsection{Guerra's interpolation scheme}
The lower bound is obtained by extending
the bound derived for $p=2$ in \cite{krzakala2016mutual}, using a Guerra-type interpolation \cite{Guerra2003}
as was already done for tensors by Korada and Macris in \cite{korada2009exact} (who consider tensors in the special case of Rademacher $P_X$). 

\begin{lemma}
	$$
	\liminf_{N \to \infty} F_N \geq \sup_{m \geq 0} \phi_{\rm RS}(m) \,.
	$$
\end{lemma}

\begin{proof}
	We use a Guerra-type interpolation \cite{Guerra2003}: Let $0 \leq t \leq 1$ and $m \in \R_+$. We suppose to observe $Y$ and $\tilde{Y}$ given by
	{
		\arraycolsep=1.4pt\def\arraystretch{1.8}
		$$
		\left\{
		\begin{array}{clll}
			Y_{i_1,\dots, i_p} &=& \displaystyle\frac{\sqrt{t (p-1)!}}{N^{(p-1)/2}} x^0_{i_1} \dots x^0_{i_p} + V_{i_1, \dots, i_p} & \quad  \text{for } 1 \leq i_1 < \dots < i_p \leq N \\
			\tilde{Y}_{j} &=& \sqrt{(1-t) m^{p-1}} x^0_j + \tilde{V}_{j} & \quad \text{for } 1 \leq j \leq N
		\end{array}
		\right.
		$$
	}
	where the variables $V_{i_1, \dots, i_p}$ and $\tilde{V}_j$ are i.i.d.\ $\mathcal{N}(0,\Delta)$ random variables. We define the interpolating Hamiltonian 
		\begin{align*}
			H_{N,t}(X)
			&= \Delta^{-1}  \sum_{i_1< \dots < i_p} \frac{\sqrt{t (p-1)!}}{N^{(p-1)/2}} Y_{i_1, \dots, i_p} x_{i_1}\dots x_{i_p} - \frac{t (p-1)!}{2N^{p-1}} (x_{i_1} \dots x_{i_p})^2 
			\\
			&+ \Delta^{-1} \sum_{j=1}^N \sqrt{(1-t) m^{p-1}} \tilde{Y}_j x_j - \frac{1}{2} (1-t) m^{p-1} x_j^2.
		\end{align*}
	Then, the posterior distribution of $X^0$ given $Y$ and $\tilde{Y}$ reads
	\begin{equation}\label{eq:interpolation_posterior}
	P(X^0=X | Y,\tilde{Y}) = \frac{1}{\cZ_{N,t}} P_X(X) \exp(H_{N,t}(X)) \,,
	\end{equation}
	where $\cZ_{N,t}$ is the appropriate normalization. Let $\psi_N(t)=\frac{1}{N} \mathbb{E}[\log \cZ_{N,t}]$ be the corresponding free energy. 
	Notice that 
	$$
	\left\{
	\begin{array}{rcl}
		\psi_N(1) &=& F_N \,, \\
		\psi_N(0)&=& \phi_{\rm RS}(m)- \frac{(1-p) m^p}{2 \Delta p} \,.
	\end{array}
	\right.
	$$
	Let $\langle \cdot \rangle_t$ denote the expectation with respect to the posterior \eqref{eq:interpolation_posterior} and let $X$ be a sampled from \eqref{eq:interpolation_posterior}, independently of everything else.

	Using Gaussian integration by parts and the Nishimori identity of Proposition~\ref{th:nishimori} one can show (see \cite{krzakala2016mutual,lelarge2016fundamental}) that for all $0 \leq t \leq 1$
	$$
	\psi_N'(t) = \frac{1}{2 \Delta p} \mathbb{E} \left\langle \left(X \cdot X^0 \right)^p - p m^{p-1}\left(X \cdot X^0 \right) \right\rangle_t +o_N(1) \,.
	$$
	By convexity of the function $a \mapsto a^p$ on $\R_+$ we have, for all $a,b \geq 0$: $a^p - p b^{p-1} a \geq (1-p) b^p$. We would like to use this inequality with $a=X\cdot X^0$ and $b=m$ to obtain that $\psi_N'(t) \geq \frac{(1-p) m^p}{2 \Delta p}$. This would conclude the proof of the lower bound because
	\begin{align*}
		\liminf_{N \to \infty} F_N = \liminf_{N \to \infty} \psi_N(1) &= \liminf_{N \to \infty} \left[\psi_N(0) + \int_0^1 \psi'_N(t) dt \right] \\
																	  &\geq \phi_{\rm RS}(m) \,.
	\end{align*}
	However, we do not know that $X \cdot X^0 \geq 0$ almost surely. 
	To bypass this issue we can add, as in sec.~\ref{sec:perturbation}, a small perturbation \eqref{eq:pert} that forces $X \cdot X^0$ concentrates around a non-negative value (Lemma~\ref{lem:overlap_concentration}), without affecting the ``interpolating free energy'' $\psi_N(t)$ in the $N \to \infty$ limit, see~\eqref{eq:pert_f}. The arguments are the same than in sec.~\ref{sec:perturbation}, so we omit the details and the rewriting of the previous calculations with the perturbation term. This concludes the proof.
\end{proof}

\subsection{Proving the upper-bound: Aizenman-Sims-Starr scheme}

We are now going to show how the arguments of \cite{lelarge2016fundamental} for the upper
bound ---using cavity computations with an Aizenman-Sims-Starr
approach~\cite{aizenman2003extended}--- can be extended to
the tensor case.

\begin{lemma}
	$$
	\limsup_{N \to \infty} F_N \leq \sup_{m \geq 0} \phi_{\rm RS}(m) \,.
	$$
\end{lemma}

\begin{proof}
	We are going to compare the system with $N$ variables to the system with $N+1$ variables.
	Define $A_N = \mathbb{E}[\log  {\mathcal Z}_{N+1} ] - \mathbb{E}[\log  {\mathcal Z}_N ]$. Consequently, $F_N = \frac{1}{N} \sum_{k=0}^{N-1} A_k$ and $\limsup F_N \leq \limsup A_N$. 

	We are thus going to upper-bound $A_N$. 
	Let $X \in S^N$ be the $N$-first variables and $\sigma\in S$ the $(N+1)^{\text{th}}$ variable. We decompose
	$
	H_{N+1}(X,\sigma) = H_N'(X) + \sigma z(X) + \sigma^2 s(X)
	$
	where
		\begin{align*}
			H_{N}'(X)
			&= \sum_{i_1< \dots < i_p}
			\frac{\Delta^{-1}\sqrt{(p-1)!}}{(N+1)^{(p-1)/2}} Y_{i_1 \dots i_p}
			x_{i_1}\dots x_{i_p} - \frac{\Delta^{-1} (p-1)!}{2(N+1)^{p-1}}
			(x_{i_1} \dots x_{i_p})^2 \,,
		\end{align*}
	\vspace{-0.7cm}
		\begin{align*}
			z(X)&=\Delta^{-1}\sum_{i_1 < \dots < i_{p-1} \leq n}  \frac{\sqrt{(p-1)!}}{(N+1)^{(p-1)/2}} Y_{i_1 \dots i_{p-1},n+1} x_{i_1} \dots x_{i_{p-1}} \,, \\
			s(X)&= - \Delta^{-1}\sum_{i_1 < \dots < i_{p-1} \leq n} \frac{(p-1)!}{2(N+1)^{p-1}} (x_{i_1} \dots x_{i_{p-1}})^2  \,.
		\end{align*}
	One can also decompose $H_N(X) = H_N'(X) + y(X)$ in law, where
	\begin{align*}
		y(X)=\Delta^{-1} \sum_{i_1< \dots < i_p} &\sqrt{(p-1)!}\left(\frac{p-1}{N^{p}} + r_n\right)^{1/2}  V'_{i_1 \dots i_p} x_{i_1} \dots x_{i_p} 
		\\
		&+
		(p-1)!\left(\frac{p-1}{N^{p}} + r_n\right) \left(x^0_{i_1} \dots x^0_{i_p} x_{i_1} \dots x_{i_p} 
		-\frac{1}{2}(x_{i_1} \dots x_{i_p})^2 \right) \,.
	\end{align*}
	In the above definition, the $V'$ are i.i.d.\ $\mathcal{N}(0,\Delta)$ random variables, independent of everything else, and $r_n = o(N^{-p})$. If we denote by $\langle \cdot \rangle'$ the Gibbs measure on $S^N$ corresponding to the Hamiltonian $\log P_X + H_N'$ we can rewrite 
	\begin{equation} \label{eq:cavity}
		A_N = \mathbb{E} \log \Big\langle \sum_{\sigma \in S} P_X(\sigma) e^{\sigma z(X) + \sigma^2 s(X)} \Big\rangle' - \mathbb{E} \log \big\langle e^{y(X)} \big\rangle' \,,
	\end{equation}
	where $X$ is a sample from $\langle \cdot \rangle'$, independently of everything else.
	$A_N$ is thus a difference of two terms that will correspond exactly to the terms of \eqref{phi_rs_1}.
	As in sec.~\ref{sec:perturbation}, we can show that under a small perturbation of the system, the overlap $X \cdot X^0$
	with the planted configuration concentrates around a non-negative value $Q'$.
	This leads to simplifications in \eqref{eq:cavity}:
	\begin{equation} \label{eq:simplification}
		\limsup_{N \to \infty} A_N \leq \limsup_{N \to \infty} \mathbb{E}[\phi_{\rm RS}(Q')]
		\leq F_{\rm RS}\,.
	\end{equation}
	For a precise derivation of \eqref{eq:simplification}, the reader is invited to report to the matrix case (see \cite{lelarge2016fundamental}, sec. 4.6), since there is no major difference with the tensor case on this point. The arguments presented there are commonly used in the study of spin glasses and are the analog of cavity computations in the SK model developed in \cite{talagrand2010mean}, sec. 1.5.
	This concludes the proof.
\end{proof}

\section{Examples of phase transitions}
\label{sec_examples}
\begin{figure*}[!h]
	\begin{center}
		\hspace*{-1.0cm}
		\includegraphics[scale=0.76]{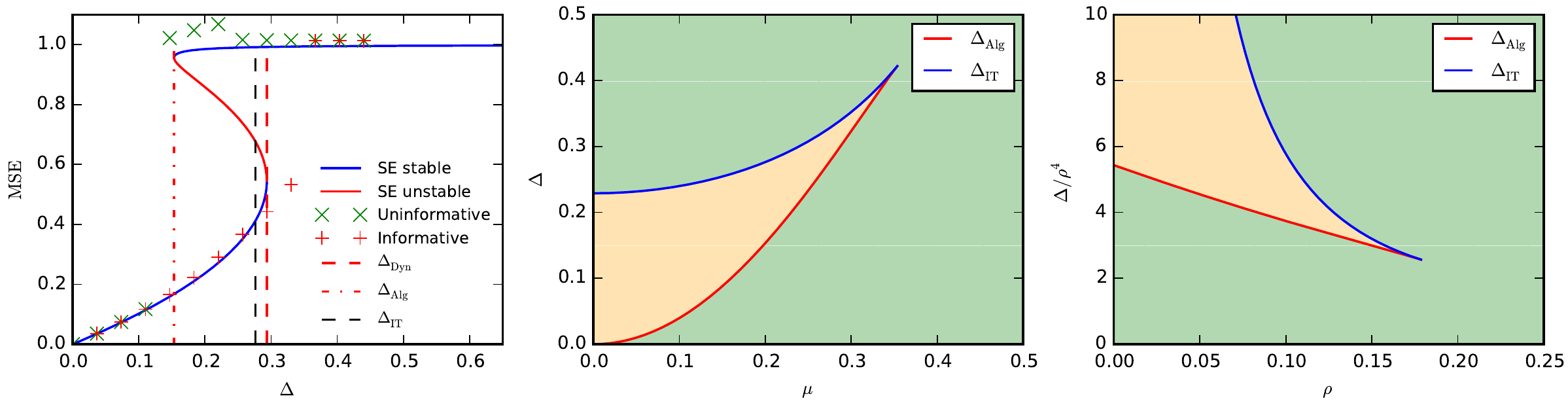}
		\caption{{\bf Left panel:} Comparison between the AMP fixed point reached from
			uninformative (marked with crosses) or informative (i.e.\ strongly correlated with the
			ground truth, marked with pluses) initialization and the fixed point of the SE equations
			(stable fixed point in blue, unstable in red). The data are for the
			Gaussian prior with mean $\mu = 0.2$, unit variance, $p=3$, $r=1$. The
			AMP runs are done on a system of size $N = 1000$.
			{\bf Central panel:} Phase diagram for the order $p=3$ tensor 
			factorization, rank $r=1$, Gaussian prior of mean $\mu$ (x-axes) and
			unit variance. In the green-shaded zone
			AMP matches the information-theoretically optimal performance, ${\rm
				MMSE}={\rm MSE}_{\rm AMP}$. In the orange-shaded zone ${\rm
			MMSE}<{\rm MSE}_{\rm AMP}$. The tri-critical point is located at $\mu_{\rm Tri} = {(p-2)}/{(2\sqrt{p-1})}$ and $\Delta_{\rm Tri} = {x_{\rm Tri}^{p-2}}/{\left(1 + x_{\rm Tri} \right)^{p-1}}$ where $x_{\rm Tri} = {(p-2)(3p - 4)}/{p^2}$.
			{\bf Right panel:} Phase diagram for the order $p=3$ tensor 
			factorization, rank $r=1$, the Bernoulli prior as a
			function of $\rho$ and $\Delta/\rho^4$. The tri-critical point is
			located at $\rho_{\rm Tri} = 0.178$ and $\Delta_{\rm Tri}/\rho^4 =
			2.60$. As $\rho\to 0$ we observed $\Delta_{\rm Alg}/\rho^4 \to
			2e$. 
			Compare to Fig. 5 in \cite{Lesieurzdeborova2017} where the same phase
			diagram is presented for the matrix factorization $p=2$ case. 
			\label{All_fig}  
		}
	\end{center}
\end{figure*}



We used the state evolution eqs.~(\ref{SE_M1}-\ref{SE_M2}), and the free
energy \eqref{BetheFreeEnergy}, to compute the values of the
thresholds $\Delta_c$, $\Delta_{\rm IT}$ and $\Delta_{\rm Alg}$ for
several examples of the prior distributions: Gaussian
(spherical spins), $P_X(x) = {\cal N}\left( \mu ,1\right)$;
Rademacher (Ising spins), $P_X(x) = \frac{1}{2}\left[\delta(x-1) +
\delta(x+1)\right]$;  Bernoulli (sub-tensor localization), $P_X(x) =
\rho \delta(x-1) + (1-\rho)\delta(x)$; and clustering (tensor stochastic block model), $P_X(x) = \frac{1}{r} \sum\limits_{k=1}^r
\delta(x - \vec{e}_k)$, where $\vec{e}_k \in \mathbb{R}^r$ is a vector with a 1 at coordinate $k$ and 0 elsewhere. Examples of values of the thresholds for the above priors are given in
Table \ref{Table_Result}. For the zero mean Gaussian and the
Rademacher prior our results for $\Delta_{\rm IT}$ indeed agree with
those presented in \cite{korada2009exact,perry2016statistical}. Central and right part
of Fig.~\ref{All_fig}
present the
thresholds for the Gaussian and Bernoulli prior as a function of the
mean $\mu$ and density $\rho$, respectively. Left part of Fig.
\ref{All_fig} illustrates that indeed the fixed points of the
state evolution agree with the fixed points of the AMP algorithm.



\begin{table}[h]
\begin{center}
  \centering
  \renewcommand{\arraystretch}{1.2}
  \begin{tabular}{|c|c|c|c|c|c|c|c|c|}
    \hline
    \backslashbox{$p$}{{Prior}} & \multicolumn{2}{c|}{{Gaussian $\cN(0,1)$}} & \multicolumn{2}{c|}{{Rademacher}} & \multicolumn{2}{c|}{{Bernoulli $\rho = 0.1$}} & \multicolumn{2}{c|}{{3 clusters}} \\
     \hline
 &
$\Delta_{\rm IT} p \log(p)$& $\Delta_{\rm Alg}$ &
$\Delta_{\rm IT}$& $\Delta_{\rm Alg}$ &
$\Delta_{\rm IT} \rho^{-p}$& $\Delta_{\rm Alg} \rho^{-2p + 2}$ &
$\frac{\Delta_{\rm IT}}{\Delta_{\rm Alg}}$&
$\frac{\Delta_{\rm Alg}r^{2p -2}}{p-1}$
\\
\hline
$  2$ &
$2 \log{2}$& $1$ &
$1$& $1$ &
$-$& $-$ &
$1$& $1$
\\
\hline
$ 3$ &
$0.754$& $0$ &
$0.2828$& $0$ &
$0.577 $& $3.738$ &
$1$& $1$
\\
\hline
$ 4$ &
$0.701$& $0$ &
$0.1902$& $0$ &
$0.398 $& $6.017$ &
$1.18$& $1$
\\
\hline
$ 5$ &
$0.685	$& $0$ &
$0.1473$& $ 0 $ &
$0.311 $& $8.251$ &
$1.62$& $1$
\\
\hline
$10$ &
$0.677$& $0$ &
$0.07216$& $0$ &
$0.154 $& $19.30$ &
$6.59$& $1$
\\
\hline
  \end{tabular}
\end{center}
\caption{Examples of the information-theoretic $\Delta_{\rm IT}$ and
  algorithmic $\Delta_{\rm Alg}$
  thresholds for order-$p$ tensor decomposition for different priors
  on the factors. For the Gaussian case $\Delta_{\rm IT}p \log(p)$
  converges to 1 at large $p$.
 For the Bernoulli case the rescaling in power of $\rho$ is for
 convenience to present quantities of order one, we did not check if it
 describes the large $p$ limit. }
\label{Table_Result}
\end{table}

\subsection{Results for Gaussian prior}
In this section we detail the analysis of the state evolution for rank $r=1$ Gaussian prior of mean $\mu$ and variance 1.
\begin{equation}
P^{\rm Gauss}_X = {\cal N}\left( \mu ,1\right) \,.
\end{equation}
Using \eqref{SE_M1} one gets for the SE equation
\begin{equation}
M^{t+1} = \frac{ \Delta\mu^2 + ({M^t})^{p-1}(1 + \mu^2)}{ \Delta + ({M^t})^{p-1}}\,,
\end{equation}
where $M$ is a scalar, and $\Delta$ is the inverse Fisher information of the output channel. 
It turns out that as soon as $p \geq 3$ the SE equation exhibits multiple stable fixed points. 

For the zero mean $\mu = 0$ case one gets
\begin{equation}
M^{t+1} = \frac{{(M^t)}^{p-1}}{\Delta + {(M^t)}^{p-1}}\,.
\end{equation}
Here the fixed point $M=0$ is stable whatever the noise $\Delta>0$ and therefore AMP will not achieve performance better than random guessing for any $\Delta>0$. Ref. \cite{richard2014statistical} studies the scaling of $\Delta$ with $N$ for which AMP and other algorithms succeed.

For positive mean $\mu >0$, however, the AMP algorithm is able to
recover the signal for values of $\Delta<\Delta_{\rm Alg}$ with 
\begin{eqnarray}
\Delta_{\rm Alg}(\mu) = \frac{x_{\rm Alg}^{p-2}}{\left(1 + x_{\rm Alg} \right)^{p-1}} , \quad \Delta_{\rm Dyn}(\mu) = \frac{x_{\rm Dyn}^{p-2}}{\left(1 + x_{\rm Dyn} \right)^{p-1}} ,
\label{Gauss_Dyn_Alg_1}
\\
x_{\rm Alg}(\mu) = \frac{p-2 + 2\mu^2 - \sqrt{(p-2)^2 - 4\mu^2(p-1)}}{2(1 + \mu^2)},
\label{Gauss_Dyn_Alg_3}
\\
x_{\rm Dyn}(\mu) = \frac{p-2 + 2\mu^2+ \sqrt{(p-2)^2 - 4\mu^2(p-1)}}{2(1 + \mu^2)},
\label{Gauss_Dyn_Alg_2}
\end{eqnarray}
where we defined a new threshold $\Delta_{\rm Dyn}$ as the smallest
such that for $\Delta > \Delta_{\rm Dyn}$ the state evolution has a
unique fixed point. 
We know of no analytical formula for $\Delta_{\rm IT}$ and for Figure \ref{All_fig} we  computed it numerically. The tri-critical point where all these curve meet is located at
\begin{eqnarray}
\mu_{\rm Tri} = \frac{p-2}{2 \sqrt{p-1}} \,.
\end{eqnarray}
Using the above expressions we derive that
\begin{eqnarray}
	\Delta_{\rm Dyn}(\mu = 0) &=& \frac{1}{p-2} \left(\frac{p-2}{p-1}\right)^{p-1} \underset{p \to \infty}{\sim} \frac{1}{e p}\, , 
\\
\Delta_{\rm Alg}(\mu) &\underset{\mu \rightarrow 0}{\sim}& \left( \frac{\mu^2}{p-2} \right)^{p-2}\, .
\end{eqnarray}
We can also compute the limit of the $\Delta_{\rm IT}(\mu=0,p)$ as $p \rightarrow \infty$ and get
\begin{equation}
	\Delta_{\rm IT}(\mu = 0,p) \underset{p \rightarrow \infty}{\sim} \frac{1}{p
  \log(p)}\, .
\end{equation}
This scaling agrees with the large $p$ behavior derived in \cite{richard2014statistical} and \cite{perry2016statistical}. 

\subsection{Results for clustering prior}
An interesting example of the prior for rank $r$ tensor estimation is 
\begin{equation}
P^{\rm Clusters}_X(x) = \frac{1}{r} \sum\limits_{1 \leq k \leq r}
\delta(x - \vec{e}_k) \, .
\label{clusters}
\end{equation}
This describes a model of $r$ non-overlapping clusters. Due to the
channel universality, this prior also describes the stochastic block
model on dense hyper-graphs as considered for sparse hyper-graph in
e.g. \cite{angelini2015spectral}. This model was considered in detail for $p=2$ in \cite{Lesieurzdeborova2017}.

The above clustering prior has non-zero mean, and it also exhibits the
transition $\Delta_c$ from a phase where recovery of clusters better
than chance is not possible, to a phase where it is.


To analyze the SE equations we first notice that the stable fixed point will be of the form
\begin{equation}
M = \frac{b I_r}{r} + \frac{(1-b) J_r}{r^2} \in \mathbb{R}^{r \times
  r} , b \in [0;1]\, ,
\end{equation}
where $I_r$ is the identity matrix and $J_r$ is a matrix filled with ones. $b = 0$ means that the estimate of the marginals does not carry any information. $b=1$ means perfect reconstruction.
The state evolution now becomes
\begin{equation}
b^{t+1} = {\cal M}_r\left( r\frac{ \left( \frac{b^{t}}{r} +
      \frac{1-b^t}{r^2}  \right)^{p-1} - \left( \frac{1-b^t}{r^2}
    \right)^{p-1}   }{\Delta}\right) \, ,\label{SE_Group}
\end{equation}
where ${\cal M}_r$ is a function that was defined and studied in
\cite{lesieur2015mmse}. Its Taylor expansion is 
\begin{eqnarray}
{\cal M}_r(x) = \frac{x}{r^2} + x^2\frac{r-4}{2 r^4} + O\left( x^3
  \right) \label{Expansion_M_r}\, .
\end{eqnarray}
We further notice that $b=0$ is always a fixed point of \eqref{SE_Group}.
By expanding \eqref{SE_Group} to first order one gets
\begin{eqnarray}
b^{t+1} = b^t \frac{p-1}{\Delta r^{2p -2}} + O\left( {b^t}^2\right) \, .
\end{eqnarray}
This fixed point therefore becomes unstable when
\begin{eqnarray}
\Delta < \Delta_c \equiv \frac{p-1}{r^{2p - 2}} \, .  \label{Deltac_clust}
\end{eqnarray}
By analyzing eq.~\eqref{SE_Group} further we can prove that $\forall x \in \mathbb{R}^+$
\begin{eqnarray}
m(x) &=& {\cal M}_r(x)
\end{eqnarray}
\begin{equation}
\Delta(x) = r\frac{ \left( \frac{{\cal M}_r(x)}{r} + \frac{1-{\cal M}_r(x)}{r^2}  \right)^{p-1} - \left( \frac{1-{\cal M}_r(x)}{r^2}  \right)^{p-1}   }{x}.
\end{equation}
$m$ is a fixed point of \eqref{SE_Group} when $m=m(x)$ and $\Delta = \Delta(x)$.
Rather than finding the fixed point iteratively, the above equations
allow us to draw all the fixed point of \eqref{SE_Group}, be it stable or unstable
We have that $m(x)$ is a stable fixed point of \eqref{SE_Group} if and only if
\begin{eqnarray}
\frac{\partial \Delta(x)}{\partial x} < 0 \,.
\end{eqnarray}
The next question is whether there is a first or second order phase
transition at $\Delta_c$. To answer this, one needs to analyze whether
the fixed point close to $b = 0$ is stable or unstable. For this we 
to compute $\frac{\partial \Delta(x)}{\partial x}$ at $x=0$ to get
using \eqref{Expansion_M_r} that 
\begin{eqnarray}
\frac{\partial \Delta(x)}{\partial x} = \frac{p-1}{2 r^{2p}} \left(-2p - r 
+ pr \right) \,.
\end{eqnarray}
Therefore if $-2p - r 
+ pr > 0$ there will be no stable fixed point close to $b=0$ and the
system must have a first order phase transition (discontinuity in the
${\rm MSE}_{\rm AMP}$) at $\Delta_c = \Delta_{\rm Alg}$. 

For two clusters $r=2$, there is a second order phase transition at
$\Delta_c$ for all $p\ge 2$. However, analyzing the state evolution
numerically we observed that for $p \ge 5$ there is a discontinuity
later at some $\Delta_{\rm Alg}< \Delta_c$. For three and more
clusters $r\ge 3$ we always have $\Delta_{\rm Alg} = \Delta_c$, and
for $-2p - r + pr \le 0$ we have not detected any other
discontinuities. Values of $\Delta_{\rm IT}$ for three clusters and
some values of $p$ are given in Table \ref{Table_Result}.

\section*{Acknowledgment}
This work has been supported by the ERC under the European Union's FP7
Grant Agreement 307087-SPARCS.

{%
	\singlespacing
	\bibliographystyle{plain}
	\bibliography{../ISITSubmited/refs}
}

\end{document}